\documentclass[12pt]{amsart}
\usepackage{amscd}
\usepackage{amssymb}
\usepackage{amstext}
\usepackage{amsthm}
\usepackage{mathrsfs}

\usepackage[final]{hyperref}
\hypersetup{unicode= false, colorlinks=true, linkcolor=blue,
anchorcolor=blue, citecolor=green, filecolor=red, menucolor=blue,
pagecolor=blue, urlcolor=blue}
\numberwithin{equation}{section}

\numberwithin{equation}{section}

\DeclareMathOperator{\Hol}{Hol}
\DeclareMathOperator{\dist}{dist}
\DeclareMathOperator{\Span}{Span}

\newcommand{\CC}{\mathbb{C}}

\renewcommand{\phi}{\varphi}

\newcommand{\ima}{{\rm Im}\,}
\newcommand{\kl}{{\Bbbk_\lambda}}

\newtheorem{Thm}{Theorem}[section]
\newtheorem{theorem}[Thm]{Theorem}
\newtheorem{lemma}[Thm]{Lemma}
\newtheorem{claim}[Thm]{Claim}

\newtheorem{corollary}[Thm]{Corollary}

\begin{document}
\sloppy
\title[Summability properties of Gabor expansions]
{Summability properties of Gabor expansions}

\author{Anton Baranov, Yurii Belov, Alexander Borichev}
\address{Anton Baranov,
\newline Department of Mathematics and Mechanics, St.~Petersburg State University, St.~Petersburg, Russia,
\newline National Research University Higher School of Economics, St.~Petersburg, Russia,
\newline {\tt anton.d.baranov@gmail.com}
\smallskip
\newline \phantom{x}\,\, Yurii Belov,
\newline St.~Petersburg State University, St. Petersburg, Russia,
\newline {\tt j\_b\_juri\_belov@mail.ru}
\smallskip
\newline \phantom{x}\,\, Alexander Borichev,
\newline I2M, CNRS, Centrale Marseille, Aix-Marseille Universit\'e, 13453 Marseille, France, 
\newline {\tt alexander.borichev@math.cnrs.fr}
}

\thanks{The work was supported by Russian Science Foundation grant 14-41-00010.}



\begin{abstract} We show that there exist complete and minimal systems 
of time-frequency shifts of Gaussians 
in $L^2(\mathbb{R})$ which are not strong Markushevich basis 
(do not admit the spectral synthesis). In particular,
it implies that there is no linear summation method for general 
Gaussian Gabor expansions. On the other hand we prove that 
the spectral synthesis for such Gabor systems holds up to one dimensional defect. 
\end{abstract}

\maketitle

\section{Introduction and the main results}

Gabor analysis is an important part of the modern time-frequency analysis. 
It deals with the expansion of functions in $L^2(\mathbb{R}^n)$ 
in the series in the time frequency shifts of a given ``window'' $\varphi$,
$$
\tau_{x,y}\varphi(s):=e^{2\pi i <y,s>}\varphi(s-x),\qquad (x,y)\in\Lambda,
$$
$\Lambda$ being a discrete subset of $\mathbb{R}^n\times\mathbb{R}^n$.

For the (most frequently used) Gaussian window 
$\gamma(s)=e^{-\pi s^2}$, these expansions are 
closely related to the corresponding uniqueness, 
sampling and interpolation problems in the Fock space, see \cite{Gro,ls}. 
The case $n=1$ corresponds here to the classical Fock space 
of one complex variable. To study expansions of
$f\in L^2(\mathbb{R}^n)$ into the series 
with the respect to the system
$$
\mathcal{G}_\Lambda(\varphi):=\{\tau_{x,y}\varphi\}_{(x,y)\in\Lambda}
$$ 
we need, first of all, the completeness property. 
On the other hand, for these expansions to be unique we 
should require the minimality property of the system 
$\mathcal{G}_\Lambda$. In 1946 Gabor considered the system 
$\mathcal{G}_{\mathbb{Z}\times\mathbb{Z}\setminus\{(0,0)\}}(\gamma)$ 
and suggested that any $f\in L^2(\mathbb{R})$ expands into a series
$$
\sum_{(x,y)\in\mathbb{Z}\times\mathbb{Z}\setminus\{(0,0)\}}c_{x,y}
\tau_{x,y}\gamma
$$
with $\ell^2$ control on $c_{x,y}$. Unfortunately, such a system 
cannot be a Riesz basis and, moreover, no system 
$\mathcal{G}_\Lambda(\gamma)$ can 
be a Riesz basis in $L^2(\mathbb{R})$, see \cite{S2}.  However, there 
are many complete and minimal systems $\mathcal{G}_\Lambda(\gamma)$. 
In particular,  $\mathcal{G}_{\mathbb{Z}\times\mathbb{Z}\setminus\{(0,0)\}}(\gamma)$ 
satisfies this property, see also \cite{als} for other constructions.

Given a complete minimal system $\mathcal{G}_\Lambda(\gamma)$ there is the (unique) biorthogonal system 
$\mathcal{H}:=\{h_{x,y}\}_{(x,y)\in\Lambda}$, 
$(\tau_{x,y}\gamma, h_{x',y'})=\delta_{x,x'}\delta_{y,y'}$, $(x,y),(x',y')\in\Lambda$. In 2015 Belov proved that such biorthogonal system is always complete \cite{b1}. Therefore, we can associate to every $f\in L^2(\mathbb{R})$ its generalized Fourier series
\begin{equation}
f\sim\sum_{(x,y)\in\Lambda}(f, h_{x,y})\tau_{x,y}\gamma,
\label{FS}
\end{equation}
and the coefficients determine the function $f$ in a unique way.

It is well known that for any linear summation method to apply to the series \eqref{FS} it is necessary that
$$f\in\Span\{(f, h_{x,y})\tau_{x,y}\gamma\}_{(x,y)\in\Lambda}.$$
This latter property is called {\it the hereditary completeness} property of the system $(\tau_{x,y}\gamma)_{(x,y)\in\Lambda}$ 
or {\it the strong Markushevich basis} property 
(or the {\it spectral synthesis} property; see \cite{BBB2} and references therein).  
It is equivalent to the completeness of every {\it mixed system}
$$
\mathcal{G}_{\Lambda_1,\Lambda_2}:=\{\tau_{x,y}\gamma\}_{(x,y)\in\Lambda_2}
\cup\{h_{x,y}\}_{(x,y)\in\Lambda_1},
$$
where $\Lambda$ is the disjoint union of $\Lambda_1$ and $\Lambda_2$.

Answering a question posed in \cite{b1} we establish the following fact.

\begin{theorem}
There exists a complete and minimal Gaussian Gabor system which is not a strong Markushevich basis.
\label{T1}
\end{theorem}

Therefore, in general, there is no linear summation method for the Gaussian Gabor systems.

Next, one may ask what is the maximal size of the orthogonal complement to the system $\mathcal{G}_{\Lambda_1,\Lambda_2}$.

\begin{theorem}
Let $\mathcal{G}_\Lambda$ be a complete and minimal 
Gaussian Gabor system. For any partition $\Lambda=\Lambda_1
\cup\Lambda_2$, $\Lambda_1\cap\Lambda_2=\emptyset$, 
the orthogonal complement to the system 
$\mathcal{G}_{\Lambda_1,\Lambda_2}$ is at most one-dimensional.
\label{T2}
\end{theorem}

In the setting of exponential systems 
on an interval a similar problem was solved in \cite{BBB}. 
It was a longstanding problem in nonharmonic Fourier analysis 
whether any complete and minimal system $\{e_\lambda\}_{\lambda\in\Lambda}$ 
in $L^2(-a,a)$, where $e_{\lambda}(t)=e^{i\lambda t}$, is hereditarily complete. 
Surprisingly the answer is the same: there exists nonhereditarily complete 
exponential systems, but the orthogonal complement to any mixed system 
is at most one-dimensional. 

The results of \cite{BBB} were generalized to systems 
of reproducing kernels in general de Branges spaces \cite{bb,BBB2} 
(note that exponential systems are unitarily equivalent to 
reproducing kernel systems in the Paley--Wiener space); 
in this case, however, the complement to a mixed system can have 
arbitrary (even infinite) dimension. At the same time,  
a full description was given in \cite{BBB2} for 
those de Branges spaces where any complete and minimal system 
of reproducing kernels is hereditarily complete. These are the  
de Branges spaces which coincide with some radial Fock spaces (the corresponding weight necessarily will have very slow growth of order at most $\exp(\log^2r)$ in contrast to the classical weight $\exp(\pi r^2)$).

While a system biorthogonal to a complete and minimal system 
of reproducing kernels in a de Branges space may have arbitrarily large defect, 
it was shown in \cite{bbb17} that in Fock-type spaces with mild regularity 
of the weight, the biorthogonal system is always complete. 
However, the conjecture that any complete and minimal system 
of reproducing kernels in a Fock space is hereditarily complete
is refuted by our Theorem \ref{T1}.

The result of \cite{BBB} may have the following Gabor-analysis interpretation. 
Consider a Gabor system $\mathcal{I}_{\mathbb{Z}\times\Lambda}$ associated with 
the window $\varphi:=\chi_{[0,1]}$, that is,
$$
\mathcal{I}_{\mathbb{Z}\times\Lambda}=
\bigl{\{}e^{2\pi i\lambda t}\chi_{(n,n+1)}(t)\bigr{\}}_{(n,\lambda)
\in\mathbb{Z}\times\Lambda}.
$$
Note that the system $\mathcal{I}_{\mathbb{Z}\times\mathbb{Z}}$ 
is an orthonormal basis in $L^2(\mathbb{R})$. 
By the results of \cite{BBB}, there exists a sequence 
$\Lambda\subset\mathbb{R}$ (which is a bounded perturbation of $\mathbb{Z}$) 
such that $\{e^{2\pi i \lambda t }\}_{\lambda\in\Lambda}$ is complete 
and minimal in $L^2(n, n+1)$ but not a strong Markushevich basis.
Hence, for any such system $\mathcal{I}_{\mathbb{Z}\times\Lambda}$ there exist 
mixed systems with any given finite or infinite defect. 

It is possible that the maximal size of the orthogonal complement to the mixed systems 
should in general depend on the time-frequency localization properties of the window function. In particular, 
it would be interesting to know whether there exists a window generating 
only strong Markushevich bases.

\subsection*{Organization of the paper.}
In Section 2 we discuss a translation of our problems to the setting 
of the classical Fock space of entire functions. In Section 3 we prove Theorem~\ref{t1} which is a reformulation of Theorem~\ref{T1}, 
while Section 4 is devoted to the proof of Theorem~\ref{codim}  which is a reformulation of Theorem~\ref{T2}.

\subsection*{Notations.} Throughout this paper the notation $U(x)\lesssim V(x)$ means that there is a constant $C$ such that
$U(x)\leq CV(x)$ holds for all $x$ in the set in question, $U, V\geq 0$. We write $U(x)\asymp V(x)$ if both $U(x)\lesssim V(x)$ and
$V(x)\lesssim U(x)$.

\subsubsection*{\bf Acknowledgments}
We are thankful to Misha Sodin for pointing   
to a version of the Ahlfors--Beurling--Carleman theorem in \cite{T}.
\bigskip

\section{Bargmann transform and the Fock space}

To translate our Gabor expansions problems into the language of entire functions we use the Bargmann transform $\mathcal{B}$:
\begin{multline*}
  \mathcal{B}f(z):=2^{1\slash4}e^{-i\pi xy}e^{\frac{\pi}{2}|z|^2}\int_\mathbb{R}f(t)e^{2\pi i yt}e^{-\pi (t-x)^2}dt \\
  = 2^{1\slash4}\int_{\mathbb{R}}f(t)e^{-\pi t^2}e^{2\pi tz}e^{-\frac{\pi}{2} z^2}dt, \quad z=x+iy.
\end{multline*}
The operator $\mathcal{B}$ maps unitarily $L^2(\mathbb{R})$ onto $\mathcal{F}$, where $\mathcal{F}$ is the classical Fock space
$$
\mathcal F=\{f\in\Hol(\CC):\|f\|^2_\mathcal F=\int_{\CC}|f(z)|^2e^{-\pi|z|^2}\,dm_2(z)<\infty\},
$$ 
$dm_2$ being planar Lebesgue measure.
Moreover, $\mathcal{B}$ maps every time-frequency shift of the Gaussian 
to a normalized reproducing kernel of $\mathcal{F}$. 
For $\lambda\in\mathbb{C}$ put 
$$
k_\lambda(z):=e^{\pi\bar{\lambda}z}.
$$
The function $k_\lambda(z)$ is the reproducing kernel for $\mathcal{F}$,
\begin{align*}
f(\lambda)&=\langle f, {k}_\lambda \rangle_{\mathcal F}, \qquad f\in\mathcal F,\\
\|{k}_\lambda\|_{\mathcal F}&=e^{\pi |\lambda|^2/2},\qquad \lambda\in\CC.
\end{align*}
It is easy to see that for $\lambda=u+iv$
$$
2^{1\slash4}\mathcal{B}(\tau_{u,v}\gamma)(z)=
e^{-\pi|\lambda|^2\slash2 }e^{\pi\lambda z}=\frac{k_{\overline 
\lambda}(z)}
{\|k_{\overline \lambda}\|_\mathcal{F}}.
$$
Here (and in what follows) we identify $\mathbb{R}^2$ with 
$\mathbb{C}$ and $(u,v)\in\mathbb{R}^2$ 
with the complex number $\lambda = u+iv$.

Thus the system $\mathcal{G}_\Lambda$ is complete and minimal 
in $L^2(\mathbb{R})$ if and only if the corresponding system of 
reproducing kernels $\{k_\lambda\}_{\lambda\in\Lambda}$ 
is complete and minimal in $\mathcal{F}$. Furthermore, 
this is equivalent to the existence of the so called generating 
function $G$ such that $G$ has simple zeros exactly at 
$\Lambda$, $g_\lambda:=G\slash(\cdot-\lambda)$ belongs 
to $\mathcal{F}$ for some (every) $\lambda\in\Lambda$ and there 
is no non-trivial entire function $T$ such that $GT\in\mathcal{F}$. 
Then the system $\{g_\lambda\slash G'(\lambda)\}_{\lambda\in\Lambda}$ 
is biorthogonal to the system  $\{k_\lambda\}_{\lambda\in\Lambda}$.

Finally, the orthogonal complement to the mixed system $\mathcal{G}_{\Lambda_1,\Lambda_2}$ is of the same dimension as the orthogonal complement to the mixed system
$$\{k_\lambda\}_{\lambda\in\Lambda_2}\cup\{g_\lambda\}_{\lambda\in\Lambda_1}.$$

Thus, our Theorems \ref{T1}, \ref{T2} can be reformulated as follows

 
\begin{theorem}\label{t1} There exists a complete minimal system of reproducing kernels $\{k_{\lambda}\}_{\lambda \in \Lambda}$ 
such that $\Lambda \subset\mathbb C$ is the disjoint union of $\Lambda_1$ and $\Lambda_2$, and the system 
$\{k_{\lambda}\}_{\lambda \in \Lambda_2}\cup \{g_{\lambda}\}_{\lambda \in \Lambda_1}$ is not complete in $\mathcal F$.
\end{theorem}

\begin{theorem}\label{codim} Let $\{k_\lambda\}_{\lambda\in\Lambda}$ be a complete and minimal system in $\mathcal{F}$, let $\{g_\lambda\}$  be its biorthogonal system, 
and let $\Lambda$ be the disjoint union of $\Lambda_1$ and $\Lambda_2$. Then the orthogonal complement to the mixed system
$$
\{k_\lambda\}_{\lambda\in\Lambda_2}\cup\{g_\lambda\}_{\lambda\in\Lambda_1}
$$
is at most one-dimensional.
\end{theorem}

One of the difficulties of dealing with the Fock space is that, 
in contrast to the de Branges spaces, it does not
possess any Riesz basis of reproducing kernels -- a tool which plays 
a crucial role in \cite{BBB, BBB2}. A good substitute
of such an orthogonal basis will be the system of the reproducing 
kernels associated with the lattice $\mathbb{Z} + i\mathbb{Z}$.

Let $\kl={k}_\lambda/\|{k}_\lambda\|$
be the normalized reproducing kernel at $\lambda$. Let $\sigma$ be the Weierstrass $\sigma$-function associated to the lattice $\mathcal{Z}=\mathbb Z+i\mathbb Z$,  
$$
\sigma(z)=z\prod_{\lambda\in\mathcal Z\setminus\{0\}}
\Bigl(1-\frac{z}{\lambda}\Bigr)e^{\frac{z}{\lambda}+\frac{z^2}{2\lambda^2}}.
$$
It is well-known that $\sigma$ is real on $\mathbb R$ and 
\begin{equation}
|\sigma(z)|\asymp\dist(z,\mathcal{Z})e^{\pi|z|^2\slash2},\qquad z\in\CC. \label{dop71}
\end{equation}
Set $\mathcal Z_0=\mathcal Z\setminus\{0\}$. 
Estimate \eqref{dop71} yields that the system $\{\Bbbk_w\}_{w\in\mathcal Z_0}$ is complete and minimal and $\sigma_0(z)=\sigma(z)/z$ is its generating function.  The system
 $\bigl{\{}\frac{\|k_w\|}{\sigma'_0(w)}\cdot\frac{\sigma_0}{\cdot-w}\bigr{\}}$ is its  biorthogonal system.
We can associate with every function $F\in\mathcal{F}$ its formal Fourier series {\it with respect to the system} $\{\Bbbk_w\}_{w\in\mathcal Z_0}$ 
by  
\begin{equation}\label{fo11}
F\sim \sum_{w\in \mathcal Z_0}\overline{b_w}\Bbbk_w,\quad\qquad b_w=
\Bigl\langle \frac{\|k_w\|}{\sigma'_0(w)}\cdot\frac{\sigma_0}{\cdot-w},F\Bigr\rangle_\mathcal{F}.
\end{equation}
Furthermore, we can write the (formal) Lagrange interpolation formula
\begin{equation}\label{fo12}
F\sim \sum_{w\in \mathcal Z_0}a_w\frac{\|k_w\|}{\sigma'_0(w)}\cdot\frac{\sigma_0}{\cdot-w},
\qquad a_w=\frac{F(w)}{\|k_w\|}.
\end{equation}
It is known that $(a_w)\in\ell^2(\mathcal Z_0)$, see e.g. \cite{S}.

\begin{lemma}[see \cite{b1}]\label{cl3}  We have
$$
\Bigl\|\frac{\sigma_0}{\cdot-w}\Bigr\|\lesssim\frac{\log^{1/2}(1+|w|)}{|w|},\qquad |b_w|^2\lesssim\log(1+|w|),\qquad w\in\mathcal Z_0.
$$
\end{lemma}

Although the series \eqref{fo11} and \eqref{fo12} do not converge for general 
$F\in\mathcal F$, our argument in Section~\ref{se3} uses the coefficients 
$(a_w)_{w\in\mathcal Z_0}$ and $(b_w)_{w\in\mathcal Z_0}$. 
In particular, we make use of the following observation from \cite[Lemma 3.1]{b1}.

\begin{lemma}
If $G$ is the generating function of a complete and minimal system of reproducing kernels $\{k_\lambda\}_{\lambda\in\Lambda}$ in $\mathcal{F}$ and $\Lambda\cap\mathcal{Z}=\emptyset$, then for 
every three distinct points $\lambda_1,\lambda_2,\lambda_3\in\Lambda$ and for every 
$F\in\mathcal{F}$ we have
\begin{multline}
\Bigl\langle\frac{G(z)}{(\cdot-\lambda_1)(\cdot-\lambda_2)(\cdot-\lambda_3)}, F{\Bigr\rangle}_{\mathcal{F}}
\\
=\sum_{w\in\mathcal{Z}\setminus\{0\}}\frac{G(w)b_w}{(w-\lambda_1)(w-\lambda_2)(w-\lambda_3)\|k_w\|},
\label{intform}
\end{multline}
with $b_w$ defined in \eqref{fo11}.
\label{l3}
\end{lemma} 
\bigskip

\section{Proof of Theorem~\ref{t1}}

\subsection{Construction} We start with an integer $u_1=Q\gg 1$, set $u_n=2^{n-1}u_1$, $n>1$, and define 
\begin{align}
\sigma_3(z)&=\frac{\sigma(z)}{z(z-1)(z-2)(z-3)},\notag\\
F(z)&=\sigma_3(z)+\sum_{n\ge 1} u_n^{-1/2}(\Bbbk_{u_n}-\Bbbk_{u_n+1})\label{fo2}.
\end{align}
Clearly, $F\in\mathcal F$. For $u=u_n$, $z\in D(u,2\sqrt{u})$ we have 
\begin{multline}
F(z)e^{-\pi|z|^2/2}=u^{-1/2}\Bigl(e^{-\pi |z-u|^2/2}-
e^{-\pi|z-u-1|^2/2+i\pi \ima z}\Bigr)e^{i\pi u\ima z}\\ + O(u^{-4}),\qquad n\to\infty,\label{fo3}
\end{multline}  
uniformly in $Q$. Furthermore, $F$ is real on $\mathbb R$.  
Therefore, for sufficiently large $Q$ and for every $n\ge 1$, there exist $\beta_n\in(1/3,2/3)$ such that $F(u_n+\beta_n)=0$, $\lim_{n\to\infty}\beta_n=\frac12$.

Set 
\begin{align*}
S(z)&=\prod_{n\ge 1}\Bigl(1-\frac{z}{u_n+\beta_n}\Bigr),\\ 
\Lambda_2&=Z_F\setminus\{u_n+\beta_n\}_{n\ge 1},
\end{align*}
where $Z_F$ is the zero set of $F$. 
Choose $v_n\in\bigl(D(u_n-\sqrt{u_n},1)\cap\mathbb R\bigr) \setminus\Lambda_2$, $n\ge 1$, and set
\begin{align*}
G_1(z)&=\prod_{n\ge 1}\Bigl(1-\frac{z}{v_n}\Bigr),\\ 
\Lambda_1&=\{v_n\}_{n\ge 1}.
\end{align*}
Next, we define 
\begin{gather*}
G_2=F/S,\quad G=G_1G_2,\quad \Lambda=\Lambda_1\cup \Lambda_2,\\
g_\lambda=\frac{G}{\cdot-\lambda},\qquad \lambda\in\Lambda.
\end{gather*}
If $Q$ is sufficiently large, then $\Lambda\cap D(0,1/2)=\emptyset$. Finally, for some $d_n\in(-1,1)$ to be chosen later on we define  
$$
H=\sigma_3+\sum_{n\ge 1}d_nu_n^{-1/3}\Bbbk_{u_n}.
$$


\subsection{Four properties}
To complete the proof of our theorem it suffices to verify the following four properties:  
\begin{itemize}
\item[(1)] $\{k_{\lambda}\}_{\lambda \in \Lambda}$ is a complete minimal system,
\item[(2)] $\langle F,k_{\lambda}\rangle_\mathcal{F}=0$, $\lambda\in\Lambda_2$,
\item[(3)] $\langle g_{\lambda},H\rangle_\mathcal{F}=0$, $\lambda\in\Lambda_1$,
\item[(4)] $\langle F,H\rangle_\mathcal{F}\not=0$.
\end{itemize}
Then, by the Hahn--Banach theorem, the system 
$\{k_{\lambda}\}_{\lambda \in \Lambda_2}\cup \{g_{\lambda}\}_{\lambda \in \Lambda_1}$ 
is not complete in $\mathcal F$.


\subsection{Estimates}
We start with the following estimates on $F$ and $G_1/S$. Since $G_1$ and $S$ are lacunary canonical products, we have 
\begin{gather*}
0<C_1\le \Bigl| \frac{G_1(z)}{S(z)}\Bigr|\cdot \Bigl| \frac{z-u_n-\beta_n}{z-v_n}\Bigr| \le C_2,\qquad 
z\in D(u_n,2\sqrt{u_n}),\,n\ge 1,\\
0<C_1\le \Bigl| \frac{G_1(z)}{S(z)}\Bigr|  \le C_2,\qquad z\in\CC\setminus \cup_{n\ge 1} D(u_n,2\sqrt{u_n}),
\end{gather*}
with $C_1,C_2$ independent of $Q\gg 1$.

Formula \eqref{fo3} implies that 
$$
C_1\le |F(z)|\sqrt{u_n}e^{-\pi|z|^2/2} \le C_2,\qquad z\in D(u_n,1/3),\,n\ge 1,
$$
with $C_1,C_2$ independent of $Q\gg 1$.
Also, since the function $e^{-\pi |z|^2/2}|\Bbbk_{u_n}(z)|$ is rather small far away from $u_n$, we can  conclude from \eqref{fo2} that,
if $Q$ is sufficiently large, then   
$$
|F(z)| \ge C_3 (1+|z|)^{-4} e^{\pi |z|^2/2}, 
$$
for
$z\in \CC \setminus\bigl(\cup_{n\ge 1}D(u_n,2\sqrt{u_n})\bigcup \cup_{a\in\mathcal Z}D(a,1/10)\bigr)$  
and for some constant $C_3>0$ independent of $Q\gg 1$. 

\begin{claim}\label{cl1} Let $\lambda\in\Lambda$. Then $g_\lambda\in\mathcal F$.
\end{claim}

\begin{proof} Our estimates on $F$ and $G_1/S$ imply that 
$$
\bigg|\frac{G_1(z)}{S(z)}\bigg| \lesssim (|z|+1)^{1/2}, \qquad 
|g_\lambda(z)| \lesssim (|z|+1)^{-1/2}|F(z)|,
$$ 
for $z\in \CC\setminus \cup_{n\ge 1} D(u_n+\beta_n,1)$. 
It remains to estimate the integrals over $\cup_{n\ge 1} D(u_n+\beta_n,1)$. Note that
$$
e^{-\pi |z|^2} \asymp |e^{-\pi z^2}|, \qquad z\in \cup_{n\ge 1} D(u_n+\beta_n,2).
$$ 
Then, by the mean value theorem, we have
\begin{gather*}
\int_{\CC} |g_\lambda(z)|^2 e^{-\pi|z|^2}\,dm_2(z) \\ \lesssim \|F\|^2_{\mathcal{F}} +
\sum_{n\ge 1} \int_{D(u_n+\beta_n,1)} |g_\lambda(z)|^2 e^{-\pi|z|^2}\,dm_2(z) \\
\lesssim 
\|F\|^2_{\mathcal{F}} +
\sum_{n\ge 1} \int_{D(u_n+\beta_n,1)} 
\Bigl|\frac{F(z)G_1(z)/S(z)}{z-\lambda}e^{-\pi z^2/2} \Bigr|^2\,dm_2(z) 
\\ \lesssim
\|F\|^2_{\mathcal{F}} +
\sum_{n\ge 1} \int_{D(u_n+\beta_n,2) \setminus D(u_n+\beta_n,1)} 
\Bigl|\frac{F(z)G_1(z)/S(z)}{z-\lambda}e^{-\pi z^2/2} \Bigr|^2\,dm_2(z) 
\\ 
\lesssim  \|F\|^2_{\mathcal{F}} +
\sum_{n\ge 1} \int_{D(u_n+\beta_n,2) \setminus D(u_n+\beta_n,1)}
|F(z)|^2 e^{-\pi |z|^2}\,dm_2(z)\lesssim  \|F\|^2_{\mathcal{F}}.  
\end{gather*}    
\end{proof}

\begin{claim}\label{cl2} For some $C>0$ independent of $Q\gg 1$ we have
$$
|\langle \sigma_3,g_\lambda\rangle_\mathcal{F}|\le \frac{C}{|\lambda|},\qquad \lambda\in\Lambda.
$$
\end{claim}

\begin{proof} For sufficiently large $Q$, both functions $F$ and $G$ have a zero 
$\lambda_0$ in the interval $(4,5)$. 
By Claim~\ref{cl1}, we have $g_{\lambda_0}\in\mathcal F$ and hence 
$$
 |G(z)|\le C(1+|z|)e^{\pi |z|^2/2},\qquad z\in\CC,
$$
with $C$ independent of $Q\gg 1$. 
Let $\lambda\in\Lambda$. Since $\Lambda\cap D(0,1/2)=\emptyset$, we have $|\lambda|\ge 1/2$. 
Then 
\begin{gather*}
   \int_{\CC}|\sigma_3(z)|\Bigl| \frac{G(z)}{z-\lambda}\Bigr| e^{-\pi|z|^2}\,dm_2(z)  
   \lesssim \int_{\CC}\Bigl|\frac{G(z)}{z-\lambda}\Bigr|
   \frac{e^{-\pi |z|^2/2}}{1+|z|^4}\,dm_2(z)  \\
   =  \int_{D(\lambda,1/8)}\ldots+\int_{D(\lambda,|\lambda|/2)\setminus D(\lambda,1/8)}\ldots+
\int_{\CC\setminus D(\lambda,|\lambda|/2)}\ldots \\
  \lesssim |\lambda|^{-1}\int_{D(\lambda,|\lambda|/2)
   \setminus D(\lambda,1/8)}\Bigl|\frac{G(z)}{z}\Bigr|\frac{e^{-\pi |z|^2/2}}{|z|^2}\,dm_2(z)  \\
   + |\lambda|^{-1}\int_{\CC\setminus D(\lambda,|\lambda|/2)}\Bigl|\frac{G(z)}{z-\lambda_0}\Bigr|\frac{e^{-\pi |z|^2/2}}{1+|z|^3}\,dm_2(z)
  \lesssim \frac{1}{|\lambda|}\,.
\end{gather*} 
Here we use again the fact that, by the mean value theorem, $\int_{D(\lambda,1/8)}\ldots 
\lesssim \int_{D(\lambda,1/4) \setminus D(\lambda,1/8)} \ldots$\,\,.
\end{proof}
\medskip


\subsection{Proof of properties (1)--(4)}

\subsection*{(1)} Let $f$ be an entire function such that $fG\in\mathcal F$. Our estimates on $F$ and $G_1/S$ imply that 
$$
|f(z)|\lesssim 1+|z|^4,\quad z\in\CC
\setminus\bigl(\cup_{n\ge 1}D(u_n,2\sqrt{u_n})\bigcup \cup_{a\in\mathcal Z}D(a,1/10)\bigr),
$$
and hence, by the maximum principle and the Liouville theorem, $f$ is a polynomial of degree at most $4$. Since
$$
|G(z)|\gtrsim e^{\pi |z|^2/2},\qquad z\in \cup_{n\ge1}D(u_n,1/3),
$$
we obtain that  $f=0$. 

Finally, by Claim~\ref{cl1}, $g_\lambda\in\mathcal F$ for every $\lambda\in\Lambda$.

Thus, we have verified that $G$ is the generating function of a complete minimal system $\{k_{\lambda}\}_{\lambda \in \Lambda}$.

\subsection*{(2)} Since $\Lambda_2\subset Z_F$, we have 
$$
\langle F,k_{\lambda}\rangle_\mathcal{F}=0, \qquad \lambda\in\Lambda_2.
$$

\subsection*{(3)} Now we are going to choose $d_n\in(-1,1)$ such that 
\begin{equation}
\langle g_{v_n},H\rangle_\mathcal{F}=0, \qquad n\ge 1. \label{fo1}
\end{equation}
We can rewrite these relations as
\begin{multline}
d_nu_n^{-1/3}\langle g_{v_n}, \Bbbk_{u_n} \rangle_\mathcal{F}
\\=-\langle g_{v_n},\sigma_3\rangle_\mathcal{F}-\sum_{m\not=n}d_mu_m^{-1/3}\langle g_{v_n}, \Bbbk_{u_m} \rangle_\mathcal{F},\qquad n\ge 1.\label{dopfo}
\end{multline}
We have
\begin{gather*}
|\langle g_{v_n}, \Bbbk_{u_n} \rangle_\mathcal{F}|=\frac{|g_{v_n}(u_n)|}{\|{k}_{u_n}\|}=\Bigl|\frac{F(u_n)G_1(u_n)}{(u_n-v_n)S(u_n)\|{k}_{u_n}\|}\Bigr|
\end{gather*}
and hence,
$$
0<C_1\le u_n^{1/2}|\langle g_{v_n}, \Bbbk_{u_n} \rangle_\mathcal{F}|\le C_2,\qquad n\ge 1.
$$
Next,
\begin{gather*}
|\langle g_{v_n}, \Bbbk_{u_m} \rangle_\mathcal{F}|=\frac{|g_{v_n}(u_m)|}{\|{k}_{u_m}\|}=\Bigl|\frac{F(u_m)G_1(u_m)}{(u_m-v_n)S(u_m)\|{k}_{u_m}\|}\Bigr|
\end{gather*}
and hence,
$$
|\langle g_{v_n}, \Bbbk_{u_m} \rangle_\mathcal{F}|\le \frac{C}{\max(u_m,u_n)},\qquad n,m\ge 1,\,n\not=m.
$$
Furthermore, by Claim~\ref{cl2},
$$
|\langle g_{v_n},\sigma_3\rangle_\mathcal{F}|\le \frac{C}{u_n},\qquad n\ge 1,
$$
with $C,C_1,C_2$ independent of $Q\gg 1$. 
Thus, we can write equalities \eqref{dopfo} as 
$$
\Delta \Xi=\Gamma,
$$
where $\Delta=(d_n)_{n\ge 1}$, $\Gamma=(\gamma_n)_{n\ge 1}$, and $\Xi=(\xi_{mn})_{m,n\ge 1}$ and
\begin{gather*}
|\gamma_n|\le Cu_n^{-1/6},\qquad n\ge 1,\\
\xi_{nn}=1,\qquad n\ge 1,\\
|\xi_{mn}|\le \frac{C}{\max(u_m,u_n)^{1/6}},\qquad n,m\ge 1,\,n\not=m,
\end{gather*}
with $C$ independent of $Q\gg 1$. Therefore, for sufficiently large $Q$ we can find $d_n\in(-1,1)$, $n\ge 1$, such that $H$ satisfies \eqref{fo1}.

\subsection*{(4)} We use that 
$$
\|F-\sigma_3\|_{\mathcal F}+\|H-\sigma_3\|_{\mathcal F}\le C Q^{-1/3}
$$
for some absolute constant $C$.
Therefore, if $Q$ is large enough, then $\langle F,H\rangle_\mathcal{F}\not=0$. \hfill \qed
\bigskip

Theorem \ref{t1} admits a reformulation in terms of weighted polynomial 
approximation in the Fock space (related to the so called Newman--Shapiro problem); it may be understood as the 
failure of a certain version of spectral synthesis in $\mathcal{F}$.  
Given a function $\phi \in \mathcal{F}$,
let us denote by $\mathcal{R}_\phi$ the subspace of $\mathcal{F}$ 
defined by  
$$
\mathcal{R}_\phi = \{f\phi\in\mathcal{F}: f\in{\rm Hol}(\mathbb C)\}.
$$
Thus, $\mathcal{R}_\phi$ is the (closed) subspace in $\mathcal{F}$ 
which consists of functions in $\mathcal{F}$ 
vanishing at the zeros of $\phi$ with appropriate multiplicities. 
Let $\mathcal{P}$ denote the set of all polynomials.

\begin{corollary}
\label{myex}
There exists $\phi\in \mathcal{F}$ such that $z^n \phi \in \mathcal{F}$ for any 
$n\ge 1$ and
$$
{\rm Clos}_{\mathcal F} \{p\phi: p\in \mathcal P\} \ne \mathcal{R}_\phi. 
$$  
\end{corollary}

\begin{proof}
Let $G=G_1G_2$ and let $\Lambda$ be the disjoint union of $\Lambda_1$ and $\Lambda_2$ as in the proof
of Theorem \ref{t1}. Denote by $H$ an element of $\mathcal F\setminus\{0\}$ 
orthogonal to the mixed system
$\{k_\lambda\}_{\lambda\in\Lambda_2}\cup\{g_\lambda\}_{\lambda\in\Lambda_1}$. 
Put $\phi=G_2$. Then $H \in \mathcal{R}_\phi$ 
and $H\perp g_\lambda$, $\lambda\in \Lambda_1$. Clearly, $z^n \phi\in\mathcal F$, $n\ge 0$. 
Let us show that $H\perp z^n \phi$, $n\ge 0$. Set $p_k(z) = \prod_{m=1}^k (1-z/v_m)$, where  
$v_m$ are the zeros of $G_1$. Then
$$
H\perp \frac{z^n G}{p_k} = \frac{z^n G_1\phi}{p_k}, \qquad k>n\ge 0.
$$
Since $G_1$ is a lacunary canonical product, it is easy to see that for every 
$\varepsilon>0$ and integer $n\ge 0$ there exists $R>0$ 
such that
$$
\int_{|z|>R} \Big|\frac{z^nG_1(z)\phi(z)\overline{H(z)}}{p_k(z)}\Big|
e^{-\pi|z|^2}\,dm_2(z) \le \varepsilon,\qquad k>n. 
$$
Since $G_1/p_k$ converges to 1 uniformly on compact sets,
we conclude that $H\perp z^n \phi$, $n\ge 0$.
\end{proof}    
\bigskip

\section{Proof of Theorem~\ref{codim}}\label{se3}

Put $d\nu(z)=e^{-\pi|z|^2}dm_2(z)$.

\begin{lemma} \label{dople} Let $F_1, F_2\in\mathcal{F}$. Define 
$$
a_w=\frac{F_2(w)}{\|k_w\|}, \qquad b_w=\Bigl\langle \frac{\|k_w\|}{\sigma'_0(w)}\cdot\frac{\sigma_0}{\cdot-w},F_1\Bigr\rangle_\mathcal{F},\qquad w\in \mathcal Z_0.
$$
Then for every $z,\mu\in\CC\setminus\mathcal Z_0$ such that 
$F_2(\mu)=0$ we have 
\begin{multline}
\sum_{w\in \mathcal Z_0}a_wb_w\Bigl[\frac{1}{z-w}+\frac{1}{w-\mu}\Bigr]\label{maineq} \\
=\int_{\mathbb{C}}\frac{\overline{F_1(\xi)}F_2(\xi)}{z-\xi}\,d\nu(\xi)-
\frac{F_2(z)}{\sigma_0(z)}\int_{\mathbb{C}}\frac{\overline{F_1(\xi)}\sigma_0(\xi)}{z-\xi}\,d\nu(\xi)+\Bigl\langle \frac{F_2}{\cdot-\mu},F_1\Bigr\rangle_{\mathcal{F}}.
\end{multline}
\end{lemma}

\begin{proof}
If $F_1=k_v$, $v\in\mathcal{Z}_0$, then
\begin{multline*}
\sigma_0(z)\sum_{w\in\mathcal{Z}_0}\frac{a_wb_w}{z-w}=\frac{\sigma_0(z)F_2(v)}{z-v} \\=\int_{\mathbb{C}}\frac{\overline{k_v(\xi)}(F_2(\xi)\sigma_0(z)-\sigma_0(\xi)F_2(z))}{z-\xi}\,d\nu(\xi).
\end{multline*}
If, additionally, $F_2(\mu)=0$, then 
$$
\sum_{w\in\mathcal{Z}_0}\frac{a_wb_w}{w-\mu}=\frac{
F_2(v)}{v-\mu} =\Bigl\langle \frac{F_2}{\cdot-\mu},k_v\Bigr\rangle_\mathcal{F}.
$$

Therefore, if $F_1$ is a finite linear combination of $k_v$, $v\in\mathcal Z_0$, and 
$F_2(\mu)=0$, 
then we have 
\begin{gather*}
\sum_{w\in \mathcal{Z}_0}a_wb_w\Bigl[\frac{1}{z-w}+\frac{1}{w-\mu}\Bigr]=
\sum_{w\in \mathcal{Z}_0}\frac{a_wb_w}{z-w}+\sum_{w\in\mathcal{Z}_0}\frac{a_wb_w}{w-\mu}\\
=\frac{1}{\sigma_0(z)}\int_{\mathbb{C}}\frac{\overline{F_1(\xi)}(F_2(\xi)\sigma_0(z)-\sigma_0(\xi)F_2(z))}{z-\xi}\,d\nu(\xi)+\Bigl\langle \frac{F_2}{\cdot-\mu},F_1\Bigr\rangle_\mathcal{F}.
\end{gather*}
On the other hand, by Claim~\ref{cl3}, for every $z,\mu\in\CC\setminus\mathcal Z_0$, the left hand side and the right hand side of \eqref{maineq} are bounded linear functionals on $F_1\in\mathcal F$. 
Since the system $\{k_v\}_{v\in\mathcal Z_0}$ is complete, the assertion of the lemma follows.
\end{proof}

We say that a measurable subset of $\CC$ is thin if it is the union of a measurable set $\Omega_1$ of zero density,
$$
\lim_{R\to\infty}\frac{m_2(\Omega_1\cap D(0,R))}{R^2}=0,
$$
and a measurable set $\Omega_2$ such that 
$$
\int_{\Omega_2}\frac{dm_2(z)}{(|z|+1)^2\log(|z|+2)}<\infty.
$$
The union of two thin sets is thin, and $\CC$ is not thin.

\begin{lemma}\label{lemmaLi} Let $f$ be an entire function of finite order, bounded on $\CC\setminus \Omega$ for some thin set $\Omega$. 
Then $f$ is a constant. 
\end{lemma}

\begin{proof}
Suppose that $f$ is not a constant and that 
\begin{equation}
\log|f(z)|=O(|z|^N),\qquad |z| \to\infty,
\label{eq1}
\end{equation}
for some $N<\infty$.
We can find $w\in\CC$ and $c\in\mathbb R$ such that the subharmonic function $u$,
$$
u(z)=\log|f(z-w)|+c
$$
is negative on $\CC\setminus \widetilde{\Omega}$ for some open  
thin set $\widetilde{\Omega}$, and $u(0)>0$. 
Given $R>0$, consider the connected component $O^R$ of $\widetilde{\Omega}\cap D(0,R)$ containing the point $0$ and set 
$$
\psi(R)=m(\partial O^R\cap \partial D(0,R)).
$$
Furthermore, set $\Delta=\{R\ge 1: \psi(R)=2\pi R\}$, $\psi_*=\psi+\infty\cdot \chi_\Delta$.

We use some estimates on harmonic measure in order to show that $\psi(R)$ should 
be not too small for many values of $R$, a contradiction to the fact that
$\widetilde{\Omega}$ is thin. 

By the theorem on harmonic estimation \cite[VII.B.1]{Koo1}, we have
\begin{equation}
u(0)\le \omega(0,\partial O^R\cap \partial D(0,R),O^R)\cdot \max_{|z|=R}u(z),
\label{eq2}
\end{equation}
where $\omega(z,E,O)$ is the harmonic measure at $z\in O$ of $E\subset \partial O$ with respect to a domain $O$. 
By the Ahlfors--Beurling--Carleman theorem \cite[Theorem III.67]{T},
\begin{equation}\label{eq3}
\omega(0,\partial\Omega^R\cap \partial D(0,R),\Omega^R)\le C\exp\Bigl(-\pi\int_1^{R/2}\frac{ds}{\psi_*(s)}\Bigr).
\end{equation}
By \eqref{eq1}--\eqref{eq3} we conclude that for some $M<\infty$
\begin{equation}\label{eq12}
\int_1^R\frac{ds}{\psi_*(s)}\le M\log R,\qquad R>2.
\end{equation}

Next, $\widetilde{\Omega}=\Omega_1\cup\Omega_2$, where $\Omega_1$ and $\Omega_2$ 
are open and
\begin{align}\label{eqstar}
\lim_{R\to\infty}\frac{m_2(\Omega_1\cap D(0,R))}{R^2}=0,\\
\label{eqstarstar}
\int_{\Omega_2}\frac{dm_2(z)}{(|z|+1)^2\log(|z|+2)}<\infty.
\end{align}
Set 
$$
\psi_j(R)=m(\Omega_j\cap \partial D(0,R)),\qquad j=1,2.
$$
We have
\begin{equation}\label{eq71}
\psi(R)\le \psi_1(R)+\psi_2(R),\qquad R>0.
\end{equation}
By \eqref{eqstar},
\begin{equation}
\label{eq31}
\int_1^R\psi_1(s)\,ds=o(R^2),\qquad R\to\infty.
\end{equation}
By \eqref{eqstarstar},
\begin{equation}
\label{eq7}
\int_2^R\frac{\psi_2(s)}{s^2\log s}\,ds<\infty.
\end{equation}

Furthermore, 
$\Delta=\Delta_1\cup\Delta_2$, where $\Delta_1$ and $\Delta_2$ 
are open and
\begin{align}\label{eqstar1}
\lim_{R\to\infty}\frac{m(\Delta_1\cap (1,R))}{R}=0,\\
\label{eqstarstar1}
\int_{\Delta_2}\frac{ds}{s\log(s+2)}<\infty.
\end{align}

Since 
\begin{multline*}
\biggl(\int_{[1,R]\setminus\Delta}\frac{ds}{s}\biggr)^2
\le \int_{[1,R]\setminus\Delta}\frac{ds}{\psi(s)}\cdot \int_{[1,R]\setminus\Delta}\!\!\!\frac{\psi(s)}{s^2}\,ds\\ \le \int_1^R\frac{ds}{\psi_*(s)}\cdot \int_1^R\frac{\psi(s)}{s^2}\,ds,
\end{multline*}
by \eqref{eq12} and \eqref{eq71} we conclude that 
\begin{equation}
\label{eq86}
\int_1^R\frac{\psi_1(s)}{s^2}\,ds+\int_1^R\frac{\psi_2(s)}{s^2}\,ds\ge \frac1{M \log R} \cdot \biggl(\int_{[1,R]\setminus\Delta}\frac{ds}{s}\biggr)^2,\, R>2.
\end{equation}

On the other hand, by \eqref{eq31} we have 
\begin{multline}\label{con1}
\int_1^R\frac{\psi_1(s)}{s^2}\,ds\le \sum_{k=0}^{[\log R]}\int_{\exp k}^{\exp(k+1)}\frac{\psi_1(s)}{s^2}\,ds\\ \le
\sum_{k=0}^{[\log R]}e^{-2k}\int_{\exp k}^{\exp(k+1)}\psi_1(s)\,ds=o(\log R),\qquad R\to\infty.
\end{multline}
Furthermore, by \eqref{eq7} we have
\begin{multline}\label{con2}
\int_e^R\frac{\psi_2(s)}{s^2}\,ds\le \sum_{k=0}^{[\log\log R]}\int_{\exp\exp k}^{\exp\exp(k+1)}\frac{\psi_2(s)}{s^2}\,ds\\ \le
\sum_{k=0}^{[\log\log R]}e^{k+1}\int_{\exp\exp k}^{\exp\exp(k+1)}\frac{\psi_2(s)}{s^2\log s}\,ds=o(\log R),\qquad R\to\infty.
\end{multline}
Therefore, \eqref{eq86}--\eqref{con2} give us that
\begin{equation*}
\int_{[1,R]\setminus\Delta}\frac{ds}{s}=o(\log R),\qquad R\to\infty.
\end{equation*}
Hence,
$$
\int_{\Delta\cap [R,R^2]}\frac{ds}{s}=(1-o(1))\log R,\qquad R\to\infty.
$$
By \eqref{eqstar1},
$$
\int_{\Delta_2\cap [R,R^2]}\frac{ds}{s}=(1-o(1))\log R,\qquad R\to\infty.
$$
Therefore, for all sufficiently large $k$, 
$$
\int_{\Delta_2\cap[2^k,2^{2k}]}\frac{ds}{s\log(s+2)}\ge \frac Ck \int_{\Delta_2\cap[2^k,2^{2k}]}\frac{ds}{s}\ge C_1,
$$
that contradicts to \eqref{eqstarstar1}.
\end{proof}

\begin{lemma}
\label{lem5} 
In the conditions of Lemma~\ref{dople}, given $z,\mu\in\CC\setminus\mathcal Z_0$ such that 
$F_2(\mu)=0$, we have 
$$
  \sum_{w\in\mathcal{Z}_0}a_wb_w\Bigl(\frac{1}{z-w}+\frac{1}{w-\mu}\Bigr)=
  \Bigl\langle \frac{F_2}{\cdot-\mu},F_1\Bigr\rangle_\mathcal{F}+\frac{\langle F_2, F_1\rangle_\mathcal{F}}{z}+o(|z|^{-1}),
$$
as $|z|\to\infty,\,z\in\CC\setminus\Omega$, for some thin set $\Omega$. 
\end{lemma}

\begin{proof}
First note that for every finite complex
measure $\psi$ on $\CC$ and $\varepsilon>0$, 
\begin{equation}
\label{meas}
\lim_{R\to\infty} R^{-2}\cdot m_2\bigg\{z\in D(0,R): 
\Bigl|\int_\mathbb{C}\frac{d\psi(\xi)}{z-\xi} - \frac{\psi(\mathbb{C})}{z} \Bigr|
\geq \frac{\varepsilon}{|z|} \bigg\} = 0.
\end{equation}
This (apparently known) fact follows, for example, from a
much more subtle result of P.~Mattila and M.~Melnikov \cite{mm} (see also \cite{ver}): 
if $\Gamma$ is a Lipschitz graph in $\CC$, $\psi$ a finite complex measure, and $\lambda>0$, 
then
$$
m_1 \bigg\{ z\in \Gamma:\, 
\sup_{\varepsilon>0} \Big| \int_{|\zeta - z|>\varepsilon} 
\frac{d\psi(\zeta)}{z-\zeta} \Big|>\lambda \bigg\} \le c_\Gamma
\frac{\|\psi\|}{|\lambda|},
$$
where $m_1$ denotes the Lebesgue measure on $\Gamma$ and
$c_\Gamma$ is a constant which depends only on the Lipschitz constant of $\Gamma$.
Applying this estimate to $\Gamma = [r e^{i\theta}, 2r e^{i\theta}]$, $r>0$, 
$\theta\in[0, 2\pi]$, 
and integrating over $\theta$ we see that 
\begin{equation}
\label{meas1}                    
m_2 \bigg\{z: r\le |z| \le 2r:\, 
\sup_{\varepsilon>0} \Big| \int_{|\zeta - z|> \varepsilon} 
\frac{d\psi(\zeta)}{z-\zeta} \Big|> \frac{\varepsilon}{r} \bigg\}\bigg| \lesssim
\frac{\|\psi\|}{\varepsilon} r^2.
\end{equation}
Now, writing $\psi = \psi_1+\psi_2$ where $\psi_1$ has compact support  
and $\|\psi_2\| \le \varepsilon^2$ and applying \eqref{meas1} to dyadic rings
$\big\{2^n\le |z| \le 2^{n+1}\}$, we easily deduce \eqref{meas}. 

By \eqref{meas},
$$
\int_{\mathbb{C}}\frac{\overline{F_1(\xi)}F_2(\xi)}{z-\xi}\,d\nu(\xi)=\frac{\langle F_2, F_1\rangle_\mathcal{F}}{z}+o(|z|^{-1}),\qquad |z|\to\infty,\,z\in\CC\setminus\Omega,
$$
for some thin set $\Omega$. 

Furthermore, by Claim~\ref{cl3}, for $z\in\mathcal Z_0$ we have
$$
\Bigl|\int_{\mathbb{C}}\frac{\overline{F_1(\xi)}\sigma_0(\xi)}{z-\xi}\,d\nu(\xi)\Bigr|
\leq\|F_1\|_{\mathcal F}\cdot
\Bigl\|\frac{\sigma_0}{\cdot-z}\Bigr\|_{\mathcal F}\lesssim\frac{\log^{1/2}(|z|+2)}{|z|+1}.
$$
It is easy to verify that this estimate extends to all $z\in\mathbb{C}$.

For every $\varepsilon>0$ the function 
$z\mapsto \frac{F_2(z)}{z\sigma_0(z)}$ belongs to $L^2\bigl(\CC\setminus \cup_{\zeta\in\mathcal Z}D(\zeta,\varepsilon)\bigr)$. 
Therefore, the set 
$$
\Omega_{\varepsilon}=\Bigl\{z\in \CC\setminus \cup_{\zeta\in\mathcal Z}D(\zeta,\varepsilon): \Bigl|\frac{F_2(z)}{\sigma_0(z)}\int_{\mathbb{C}}\frac{\overline{F_1(\xi)}\sigma_0(\xi)}{z-\xi}\,d\nu(\xi)\Bigr|>\frac{\varepsilon}{|z|}\Bigr\}
$$ 
satisfies the estimate 
$$
\int_{\Omega_\varepsilon}\frac{dm(z)}{(|z|+1)^2\log(|z|+2)}<\infty, 
$$
and hence is thin.
As a result, 
$$
\frac{F_2(z)}{\sigma_0(z)}\int_{\mathbb{C}}\frac{\overline{F_1(\xi)}\sigma_0(\xi)}{z-\xi}\,d\nu(\xi)=o(|z|^{-1}),\qquad |z|\to\infty,\,z\in\CC\setminus\Omega,
$$
for some thin set $\Omega$. 
\end{proof}

\begin{proof}[Proof of Theorem~\ref{codim}] Without loss of generality we can assume that $\Lambda_1$ and $\Lambda_2$ are infinite (otherwise a linear algebra argument shows that the mixed system is complete). Choose two entire functions of finite order $G_1$ 
and $G_2$ with simple zeros, correspondingly, at $\Lambda_1$ and $\Lambda_2$ 
such that $G=G_1G_2$.
We fix four distinct points $\lambda_1, \lambda_2,\lambda_3, \lambda_4$ such that $\lambda_1,\lambda_2\in\Lambda_1$, $\lambda_3,\lambda_4\in\Lambda_2$. 

\subsection*{Step 1.} Let us start with two vectors $H_1,H_2\in\mathcal F$ such that    
$H_2\perp\{k_\lambda\}_{\lambda\in\Lambda_2}$, $H_1\perp\{g_\lambda\}_{\lambda\in\Lambda_1}$. 

Then $H_2=G_2S_2$ for some entire function $S_2$ (independent of $H_1$).
Denote by $\Sigma$ the zero set of $S_2$. If $\Sigma$ is finite, then $S_2=Pe^{\varphi}$ 
for an entire function $\varphi$
and a polynomial $P$. In this case we can multiply $G_2$ by $e^{\varphi}$ and $G_1$ and $S_2$ by $e^{-\varphi}$.
So in the case when $\Sigma$ is finite we can assume that $S_2$ is a polynomial. 

Applying, if necessary, 
the Weyl translation operator $W_a:F\mapsto F(\cdot+a)e^{-\pi\bar{a}\cdot}$ (see, for instance, \cite{S}), we can guarantee that either 
$\Sigma$ is finite or it contains an infinite subset 
separated from $\mathcal Z$ and that the zeros of $G$ and $H_2$ are disjoint from $\mathcal Z$. 

We have 
\begin{equation}
\label{dop7f}
0=\langle g_\lambda,H_1\rangle_\mathcal{F}=\Bigl\langle \frac{G}{\cdot-\lambda},H_1\Bigr\rangle_\mathcal{F},\qquad \lambda\in\Lambda_1.
\end{equation}
By Lemma \ref{l3}, for $\lambda\in\Lambda_1\setminus\{\lambda_1,\lambda_2\}$ we have 
\begin{multline*}
0=\Bigl\langle \frac{G}{(\cdot-\lambda)(\cdot-\lambda_1)(\cdot-\lambda_2)},H_1\Bigr\rangle_\mathcal{F} 
\\=
\sum_{w\in \mathcal{Z}_0}\frac{b_wG(w)}{\|k_w\|(w-\lambda)(w-\lambda_1)(w-\lambda_2)},
\end{multline*}
where 
$$
b_w=\Bigl\langle \frac{\|k_w\|}{\sigma'_0(w)}\cdot\frac{\sigma_0}{\cdot-w},H_1\Bigr\rangle_\mathcal{F}.
$$

Set
$$
L(z)=\sum_{w\in \mathcal{Z}_0}\frac{b_wG(w)}{\|k_w\|(z-w)(\lambda_1-w)(\lambda_2-w)}.
$$ 
By \eqref{dop7f}, $L(\lambda)=0$, $\lambda\in \Lambda_1\setminus\{\lambda_1,\lambda_2\}$.

Therefore, for some entire function $S_1$ (independent of $H_2$) we have 
\begin{equation}
\sum_{w\in \mathcal{Z}_0}\frac{b_wG(w)}{\|k_w\|(z-w)(\lambda_1-w)(\lambda_2-w)}=\frac{G_1(z)S_1(z)}{\sigma_0(z)(z-\lambda_1)(z-\lambda_2)}.
\label{e1}
\end{equation}
On the other hand, we have the following interpolation formula for 
$H_2/[\sigma_0(\lambda_3-\cdot)(\lambda_4-\cdot)]$: 
\begin{equation}
\label{e2}
\sum_{w\in \mathcal Z_0}\frac{a_w\|k_w\|}{\sigma'_0(w)(z-w)
(\lambda_3-w)(\lambda_4-w)} = \frac{G_2(z)S_2(z)}{\sigma_0(z)(z-\lambda_3)(z-\lambda_4)},
\end{equation}
where $a_w=H_2(w)/\|k_w\|$, $w\in \mathcal Z_0$. 
Indeed, the difference between the two parts of this equality is an entire function tending to zero 
outside any $\varepsilon$-neighborhood of $\mathcal Z_0$. 
Set 
$$
M(z)=\sum_{w\in \mathcal Z_0}\frac{a_w\|k_w\|}{\sigma'_0(w)(z-w)
(\lambda_3-w)(\lambda_4-w)}.
$$
Calculating the residues in \eqref{e1} and \eqref{e2} 
we obtain 
$$
S_1(w)=\frac{b_w\sigma_0'(w)G_2(w)}{\|k_w\|},
\qquad S_2(w)=\frac{a_w\|k_w\|}{G_2(w)},\qquad w\in\mathcal Z_0.
$$
Hence, for every $\mu\in\CC\setminus \mathcal Z_0$, there exists an entire function $T_\mu$ such that
\begin{multline}
\label{Teq}
\frac{S_1(z)S_2(z)}{\sigma_0(z)}=
\sum_{w\in \mathcal{Z}_0}a_wb_w\Bigl[\frac{1}{z-w}+\frac{1}{w-\mu}\Bigr]+T_\mu(z)
\\= \mathcal C_\mu(z) + T_\mu(z),
\end{multline}
with 
$$
{\mathcal C}_\mu(z) = \sum_{w\in \mathcal{Z}_0}a_wb_w\Bigl[\frac{1}{z-w}+\frac{1}{w-\mu}\Bigr].
$$

\subsection*{Step 2.} Next, we show that $T_\mu$ is a polynomial for every $\mu\in\CC\setminus \mathcal Z_0$.  Indeed, we can multiply \eqref{e1} by \eqref{e2} and obtain 
\begin{multline*}
L(z)M(z)=
\frac{G(z)S_1(z)S_2(z)}{\sigma^2_0(z)(z-\lambda_1)(z-\lambda_2)(z-\lambda_3)(z-\lambda_4)}\\
=\frac{G(z)\mathcal C_\mu(z)}{\sigma_0(z)(z-\lambda_1)(z-\lambda_2)(z-\lambda_3)(z-\lambda_4)}\\ +
\frac{G(z)T_\mu(z)}{\sigma_0(z)(z-\lambda_1)(z-\lambda_2)(z-\lambda_3)(z-\lambda_4)}.
\end{multline*}
Since $L(z)$, $M(z)$, and $C_\mu(z)/|z|^2$ are $O(1/|z|)$ as $|z|\to\infty$ while $\dist(z,\mathcal Z)>\varepsilon$, 
it is easy to see that 
$$
\frac{GT_\mu}{(\cdot-\lambda_1)(\cdot-\lambda_2)(\cdot-\lambda_3)(\cdot-\lambda_4)}\in\mathcal{F}.
$$ 
If $T$ has at least $4$ zeros $t_1,\ldots,t_4$, then 
$$
\frac{GT_\mu}{(\cdot-t_1)(\cdot-t_2)(\cdot-t_3)(\cdot-t_4)}\in\mathcal{F}. 
$$
This contradicts to the completeness of the system $\{k_\lambda\}_{\lambda\in\Lambda}$. If $T_\mu$ is not a polynomial and has finite
number of zeros, then $T_\mu-1$ has infinite number of zeros and we arrive again at a contradiction. Therefore, $T_\mu$ is a polynomial. 

\subsection*{Step 3.} Now, we will show that $T_\mu$ is a constant for every $\mu\in\CC\setminus \mathcal Z_0$. Indeed, otherwise, the estimate
$$
\sum_{w\in \mathcal{Z}\setminus\{0\}}a_wb_w\Bigl[\frac{1}{z-w}+\frac{1}{w-\mu}\Bigr]=o(|z|), \qquad |z|\to\infty,\,\dist(z,\mathcal{Z})>\varepsilon,
$$
yields that all zeros of $S_2$ of large modulus are close to $\mathcal{Z}$ and, hence, 
by the above remarks, $S_2$ is a polynomial. Then by \eqref{Teq}, 
$$
|S_1(z)|\gtrsim \frac{|\sigma_0(z)|}{1+|z|^N},\qquad \dist(z,\mathcal{Z})>\varepsilon,
$$ 
for some $N<\infty$. By \eqref{e1} we obtain that $G_1$ is of at most polynomial growth, 
which is impossible since $\Lambda_1$ is infinite.

\subsection*{Step 4.} Thus, for every $\mu\in\CC\setminus \mathcal Z_0$, the function $T_\mu$ is a constant.  
Hence, by Lemma~\ref{lem5}, for every zero $\mu$ of $H_2$ we have 
\begin{gather}
\label{tempor} \frac{S_1(z)S_2(z)}{\sigma_0(z)} =\sum_{w\in \mathcal{Z}_0}a_wb_w\Bigl[\frac{1}{z-w}+\frac{1}{w-\mu}\Bigr]+T_\mu \\
 = \Bigl\langle \frac{H_2}{\cdot-\mu},H_1\Bigr\rangle_\mathcal{F}+T_\mu+\frac{\langle H_2,H_1\rangle_\mathcal{F}}{z}+o(|z|^{-1}),\quad |z|\to\infty,\,z\in\CC\setminus\Omega,\notag
\end{gather}
for some thin set $\Omega$. 

If $S_2$ is a polynomial, and 
$$
\Bigl\langle \frac{H_2}{\cdot-\mu},H_1\Bigr\rangle_\mathcal{F}+T_\mu\not=0
$$
for a zero $\mu$ of $H_2$, then it follows from \eqref{tempor} that
$$
|S_1(z)|\gtrsim\frac{|\sigma_0(z)|}{1+|z|^N}, \qquad z\in\mathbb{C}\setminus\Omega, 
$$
for some thin set $\Omega$. By \eqref{e1} we obtain that $G_1$ 
is of at most polynomial growth on $\mathbb C\setminus\widetilde{\Omega}$ with 
another thin set $\widetilde{\Omega}=\Omega\cup\bigcup_{z\in \mathcal{Z}_0}D(z,1/|z|)$. 
By Lemma~\ref{lemmaLi} we obtain that $G_1$ is a polynomial, which is impossible.

Suppose now that $S_2$ is not a polynomial, and then its zero set, $\Sigma$, is infinite. 
By \eqref{Teq}, for every $\mu\in\Sigma$ we have $T_\mu=0$. 
As a result, the values $\langle H_2/(\cdot-\mu),H_1\rangle_\mathcal{F}$ 
do not depend on $\mu\in\Sigma$. Using that 
$\lim_{\mu\to\infty}\|H_2\slash(\cdot-\mu)\|_\mathcal{F}=0$, we conclude that $\langle H_1, H_2/(\cdot-\mu)\rangle_\mathcal{F}=0$, $\mu\in\Sigma$.

Summing up, we always have 
\begin{equation}\label{fo7}
\frac{S_1(z)S_2(z)}{\sigma_0(z)}=\frac{\langle H_2,H_1\rangle_\mathcal{F}}{z}+o(z^{-1}),\qquad |z|\to\infty,\,z\in\CC\setminus\Omega,
\end{equation}
for some thin set $\Omega$.

\subsection*{\bf Step 5.} Now, suppose that the dimension of the orthogonal complement to the mixed system
$$
\{k_\lambda\}_{\lambda\in\Lambda_2}\cup\{g_\lambda\}_{\lambda\in\Lambda_1}
$$
is at least two. 
Then we can choose two vectors $H,\tilde H\in\mathcal F$ 
such that
\begin{gather*}
H,\tilde H \perp\{g_\lambda\}_{\lambda\in\Lambda_1}\cup 
\{k_\lambda\}_{\lambda\in\Lambda_2},\qquad 
\|H\|=\|\tilde H\|=1,\qquad \langle H,\tilde H\rangle_\mathcal{F}=0.
\end{gather*}

Let $S_1,S_2$ be the $S$-functions corresponding to $H$, and let $\tilde S_1$, $\tilde S_2$ be the $S$-functions corresponding to $\tilde H$. 
Applying the previous argument to the pairs $(H_1,H_2)=(H,H),(\tilde H,\tilde H),
(H,\tilde H), (\tilde H,H)$ we obtain similarly to \eqref{fo7} that 
\begin{align*}
A&=\frac{S_1(z)S_2(z)}{\sigma_0(z)}=\frac{1}{z}+o(|z|^{-1}),\qquad |z|\to\infty,\,z\in\CC\setminus\Omega,\\
B&=\frac{\tilde{S_1}(z)\tilde{S}_2(z)}{\sigma_0(z)}=\frac{1}{z}+o(|z|^{-1}),\qquad |z|\to\infty,\,z\in\CC\setminus\Omega,\\
C&=\frac{S_1(z)\tilde{S}_2(z)}{\sigma_0(z)}=o(|z|^{-1}),\qquad |z|\to\infty,\,z\in\CC\setminus\Omega,\\
D&=\frac{\tilde{S}_1(z)S_2(z)}{\sigma_0(z)}=o(|z|^{-1}),\qquad |z|\to\infty,\,z\in\CC\setminus\Omega,\\
\end{align*}
for some thin set $\Omega$.

The identity $AB=CD$ gives us a contradiction.
\end{proof}

\end{document}